\documentclass[12pt]{article}

\usepackage{fancyhdr, amsmath, amssymb, amsfonts, url, subfigure, dsfont, mathrsfs, graphicx, epsfig, subfigure, amsthm, tikz}
    \usepackage{epsfig}
    \usepackage{graphicx}
    \usepackage{color}

\newlength{\guillotine}
\newtheorem{thm}{Theorem}[section]

\newtheorem{lemma}[thm]{Lemma}

\newtheorem{prop}[thm]{Proposition}

\newtheorem{definition}[thm]{Definition}
\newtheorem{example}[thm]{Example}

\theoremstyle{remark}
\newtheorem{rem}[thm]{Remark}
\begin{document}

\title{Volume growth for infinite graphs and translation surfaces}

\author{P. Colognese and M. Pollicott\thanks{Department of Mathematics, Warwick University, Coventry, CV4 7AL, UK}}
\date{
   Warwick University\\[2ex]
    \today
}

\maketitle

\section{Introduction}

We begin by recalling the  definition of volume entropy
 for compact Riemannian manifolds due to Manning \cite{manning}.
Let $M$ be a compact  manifold 
with Riemannian metric $\rho$ and universal cover 
$\widetilde M$ equipped with the lifted metric $\widetilde \rho$.  
  Fix a point $ c \in \widetilde M$ and consider a ball
  $B(c, R)$  
   of radius $R>0$  centred at
  $ c$.

\begin{definition}\label{entropydef}
The {\it volume entropy} of $M$ is defined by 
$$
h = h(M) := \lim_{R \to \infty} \frac{1}{R} \log
\hbox{\rm Vol}_{\tilde{\rho}}(B( c, R)),
$$
where $\hbox{\rm Vol}_{\tilde{\rho}}$ denotes the Riemannian volume on $\widetilde M$ with respect to $\widetilde \rho$.
\end{definition}

For manifolds $(M, \rho)$ of non-positive curvature this  coincides with the topological entropy $h$ of the associated geodesic flow \cite{manning}.
In the case of manifolds with negative sectional curvature,  Margulis \cite{margulis} showed in his thesis that there is a simple asymptotic formula:  There exists $C > 0$ such that
$$
\lim_{R \to +\infty} \frac{\hbox{\rm Vol}_{\tilde{\rho}}(B( c, R))}{e^{h R}} = C.
$$
A closely related result in \cite{margulis} gave an asymptotic formula for the number $\Pi(x,R)$ of geodesic arcs starting and finishing at a given point $x$ of length at most $R$:  
There exists $D > 0$
such that
$$
\lim_{R \to +\infty} \frac{\Pi(x, R)}{e^{h R}} = D.
$$

 A related notion of volume entropy was considered for directed, finite, connected, non-cyclic graphs without terminal vertices by  Lim in \cite{lim}. 
In this note we  extend Lim's definition of volume entropy to suitable infinite graphs 
and show the analogue of Margulis' result in this context (Theorem \ref{asyminfinite}). 
 As an application we show a version of Margulis' theorem for the   natural analogue of  volume growth for translation surfaces (Theorem \ref{trans}).

  This note originated as a summer MPhil project of the first author.
It may have been possible to apply the transfer operator  methods in \cite{PU}, but instead we employ a more direct and elementary  approach.
    
    We are grateful to A. Eskin, J. Chaika, R. Sharp, S. Ghazouani and the three anonymous referees for their useful comments.

\section{Infinite Graphs}
In this section we will introduce the types of graphs we shall we working with as well as basic definitions which will be used throughout the paper.

Let $\mathcal G$ be a non-empty connected oriented graph. Let $\mathcal V = \mathcal V(\mathcal G)$ and $\mathcal E = \mathcal E(\mathcal G)$ be the vertex and oriented edge sets respectively. For every edge $e$, let $i(e)$ and $t(e)$ denote the initial and the terminal vertex of $e$, respectively. We 
can define a length distance $d$ on $\mathcal G$ by 
introducing a length function $\ell: \mathcal E \to \mathbb R$ which
assigns a positive real number 
$\ell(e)$ to each edge $e \in \mathcal E$.

\begin{example}[Infinite Graph] \label{example}
Consider a  graph $\mathcal G$ formed from one vertex  and a countably infinite number of edges.
\begin{figure}[ht]
\begin{center}
\begin{tikzpicture}[scale=0.75]
  \draw [fill, black] (-5,0) circle (0.21);
   \draw [->] plot [smooth, tension=2] coordinates { (-5.2,0.25) (-3.5,4) (-1,0) (-3.5,-4) (-5.2,-0.25)};
 \draw [->] plot [smooth, tension=2] coordinates { (-5.1,0.25) (-3.5,3) (-2,0) (-3.5,-3) (-5.1,-0.25)};
  \draw [->] plot [smooth, tension=2] coordinates { (-5,0.25) (-4,2) (-3,0) (-4,-2) (-5,-0.25)};
 \draw [->] plot [smooth, tension=2] coordinates { (-4.9,0.25) (-4.5,1) (-4,0) (-4.5,-1) (-4.9,-0.25)};
 \node at (0.5,0) {$\cdots$};
 \node at (-5.5,0) {$v$};
  \node at (-3.7,0) {$e_1$};
    \node at (-2.7,0) {$e_2$};
       \node at (-1.7,0) {$e_3$};
       \node at (-0.7,0) {$e_4$};
 \end{tikzpicture}
 \end{center}
 \caption{ A single vertex $\mathcal V = \{v\}$ and 
 infinitely many edges $\mathcal E = \{e_n\}_{n=1}^\infty$.}

\end{figure}
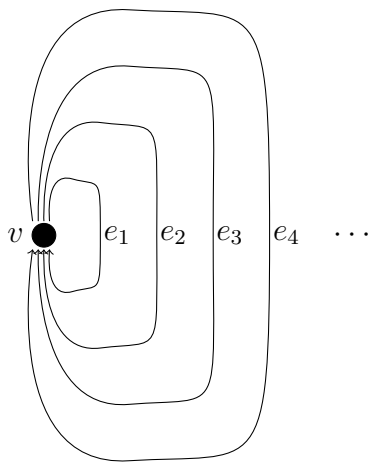
\end{example}

A {\it path}  in $\mathcal G$ corresponds 
to  a sequence of edges  $p=e_1\ldots e_n$  for which  $t(e_j)=i(e_{j+1})$, for $1\leq j <n$ and we denote its length  by $\ell(p)=\sum_{j=1}^n \ell(e_j)$. 

\medskip
\noindent
Let $\mathcal P_{{\mathcal G}}(x,R) = \{p = e_1\ldots e_n \hbox{ : } i(e_1)=x,  \ell(p) \leq R\}$ denote the set of all such paths of length at most  $R$  starting at $x\in \mathcal V({\mathcal G})$.  
We denote its cardinality by 
$N_{{\mathcal G}}(x,R) = \#\mathcal P_{{\mathcal G}}(x,R)$. 
\medskip

\begin{definition}\label{gentropydef}
We define the \textit{volume entropy} of $(\mathcal G, \ell,x)$ as
\[h(\mathcal G, \ell,x)=\limsup_{R\rightarrow +\infty}\frac{1}{R}\log N_{\mathcal G}(x,R).\]
\end{definition} 

However, we need to make further assumptions on the length function $\ell$ for 
$h(\mathcal G, \ell,x)$ to be finite. 
To see this, consider
the graph  $\mathcal G$ 
in 
Example \ref{example} 
which has a single vertex  and an infinite number of edges, and assume that the  lengths don't tend to infinity.
Then for $R$ sufficiently large, $N_{\mathcal G}(x,R)=\infty$ and thus $h(\mathcal G, \ell,x) = \infty$.

We summarise below the properties of the graph that are needed in the proof.

\bigskip
\noindent
{\bf Graph Hypotheses.}
Henceforth, we shall consider graphs with finite vertex set $\mathcal V$ and a countable edge set $\mathcal E$. Furthermore we require that $\mathcal E$ and  the associated
length function satisfy the following properties:
\begin{enumerate}
    \item[(H1)] For all $\sigma>0$ we have 
    $\sum_{e \in \mathcal E} e^{-\sigma \ell(e)}<\infty$;
    \item[(H2)] 
     For all  edges $e,e'\in \mathcal E$ there exists a path in $\mathcal G$ which starts with $e$ and ends with $e'$\footnote{A slightly  weaker  assumption would be to  require that for a sufficiently large finite subset
$\mathcal E_0 \subset \mathcal E$, for  every  
     $e, e' \in \mathcal E_0$ there exists a path in $\mathcal G$ which starts with $e$ and ends with $e'$}; 
and  
    \item[(H3)]  There does not exist a $d > 0$
    such that
     $$\{\ell(c) \hbox{ : } c \hbox{ is a closed path}\}
\subset d \mathbb N.     
     $$
\end{enumerate}
\medskip
\noindent
Under the above hypotheses, the  volume entropy $h=h(\mathcal G, \ell,x)$ does not depend on the choice of  base point $x$.


\begin{lemma}
If 
the  graph $\mathcal G$ satisfies  (H1) and (H2) then $0<h<\infty$. 
\end{lemma}
\begin{proof}
By assumption (H2), and the pigeonhole principle applied to $\mathcal V$, there exist a path connecting the base point $x$ to some vertex $v$  and two closed paths, $c_1$ and $c_2$, which pass through $v$. By considering all possible concatenations of these closed paths it is clear that 
there exists $b>0$ such that 
 $N_{\mathcal{G}}(x,R)\geq 2^{\lfloor R/b\rfloor}$ for all $R>0$ and hence $h\geq \frac{\log 2}{b} >0$.

To see that $h$ is finite 
we can formally write
$$\sum_{p\in \mathcal{P}_{\mathcal{G}}(x,R)}e^{-\sigma \ell(p)}\leq \sum_{n=1}^\infty \bigg(\sum_{e\in \mathcal{E}}e^{-\sigma \ell(e)}\bigg)^n,\eqno(2.1)$$
for $\sigma > 0$,
where the Right Hand Side involves all possible sums of edge lengths.
Using (H1) one can see that 
for $\sigma = \sigma_0$
sufficiently large  $\sum_{e\in 
\mathcal{E}}e^{-\sigma \ell(e)} < 1$
and thus the geometric series 
on the Right Hand Side of (2.1) converges.
 In particular, since $h$ is 
easily seen to be  
 the absicssa of convergence of the series on the Left Hand Side of (2.1) we see  that
$h \leq \sigma_0 < +\infty$, as required.
\end{proof}

Our main result for $\mathcal G$ is  the following asymptotic for the growth of paths.
 \begin{thm}\label{asyminfinite}
If 
the  graph $\mathcal G$ satisfies  (H1),(H2) and (H3) then
there exists a constant $C>0$ such that $N_{{\mathcal G}}(x,R) \sim C e^{hR}$, i.e.,
 $$
 \lim_{R \to +\infty}
 \frac{N_{{\mathcal G}}(x,R)}{e^{hR}} = C.
 $$ 
 \end{thm}
 
 The  proof follows the lines of the classical proof of the prime number theorem.
 In particular, it is  based on the use of a Tauberian theorem (in Section 5).
 This, in turn,  depends on the properties of the complex function $\eta_{\mathcal{G}}(z)$, the Laplace transform of $N_{{\mathcal G}}(x,R)$ (defined in Section 4).  The function $\eta_{\mathcal{G}}(z)$   is   analysed using matrices introduced in the next section.  In the special case of finite graphs, the asymptotic in Theorem \ref{asyminfinite} could be easily deduced using ideas in \cite{parry} for finite matrices.  
 
 \begin{rem}
 Without hypothesis  (H3) this theorem may not hold.   For example, even in the case of finite graphs, if we consider the graph $\mathcal G$ with a single vertex and two edges of length 1, then $N_{\mathcal G}(x,R)=2^{\lfloor R \rfloor}$ for all $R>0$. In this case, the  limit in Theorem 2.4 does not converge.

 \end{rem}
 \section{Countable Matrices}

In this section consider a 
graph  
$\mathcal G$ 
 and length function $\ell$
which satisfy hypotheses (H1)-(H3). 
Let us order the edge set $\mathcal E = (e_a)_{a \in \mathbb N}$ by non-decreasing length and write $\ell (a) := \ell(e_a)$, 
$a \in \mathbb N$.

\begin{definition}
 We can associate to $\mathcal G$ the infinite matrix $M_0$ 
 defined by 
$$M_0(a,b)=\begin{cases} 1 &\mbox{if }  t(a)=i(b), \\
0 & \mbox{otherwise. }  \end{cases} $$
For each  $z\in \mathbb C$ 
we define the matrix $M_z$ by $M_z(a,b)=M_0(a,b)e^{-z\ell(b)}$ for $a,b\in \mathcal E$.\\
\end{definition}



Let $P(n,a,b)$ denote the set of paths in $\mathcal {G}$ consisting of $n$ edges, 
starting with edge $e_a$ and ending with edge $e_b$. It then follows from 
formal matrix multiplication that for any $n\geq 1$, we can write the $(a,b)^{th}$ entry of the $n^{th}$ power of the matrix as:
$$
M_z^n(a,b) = e^{z\ell(a)}\sum_{p\in P(n+1,a,b)} e^{-z\ell(p)}, \eqno(3.1).
$$
which will be finite by hypothesis  (H1).

Given a matrix $L = \left(L(a,b)\right)_{a,b =1}^\infty$ with $\sup_a \sum_b |L(a,b)| < +\infty$ we can associate to $L$ 
a bounded linear operator 
$\widehat L: \ell^\infty(\mathbb C) \to \ell^\infty(\mathbb C)$
by
$$\widehat L(\underline u)=\Big(\sum_{b=1}^\infty L(a,b)u_b\Big)_{a=1}^\infty \hbox{ where }
\underline u = (u_b)_{b=1}^\infty \in \ell^\infty(\mathbb C).$$
In particular, by hypothesis (H1), when $Re(z)>0$ we can associate to  $M_z$  a bounded operator  $\widehat {M_z}: \ell^\infty(\mathbb{C}) \to \ell^\infty(\mathbb{C})$ by
$$\widehat {M_z}(\underline u)=\Big(\sum_{b=1}^\infty M_z(a,b)u_b\Big)_{a=1}^\infty.$$

To proceed, we would like to understand the domain of meromorphicity  of the linear operator $(I- \widehat {M_z})^{-1}: \ell^\infty(\mathbb C) \to \ell^\infty(\mathbb C)$, where $I$ denotes the identity operator.
To this end, we shall make use of an idea by Hofbauer and Keller in \cite{hofbauerkeller}, where they observe that the invertibility of certain operators
of the above form depends only on the determinant of an associated finite matrix.\\

Fix  $\epsilon > 0$ and, for convenience,  assume also $h > \epsilon$. 
Given $k \geq 1$, 
we can truncate the matrix $M_z$ to the $k \times k$ matrix  $A_z = (M_z(i,j))_{i,j=1}^k$.
Then  we can then write
$$
M_z = 
\left(
\begin{matrix}
A_z&B_z  \\ C_z &D_z
\end{matrix}
\right)
$$
where, in particular, 
$D_z = \left(M_z(i+k,j+k) \right)_{i,j=1}^\infty$. Again, we can interpret $I-\widehat{D_z}$ as a bounded linear operator  on $\ell^\infty(\mathbb C)$ 
and write   $(I-\widehat{D_z})^{-1}=\sum_{m=0}^\infty \widehat{D_z}^m$ if the operator $\widehat{D_z}$ has norm $\|\widehat{D_z}\|<1$.
In particular, this is true when  $Re(z) \geq \epsilon$ for $k$ sufficiently large, since by (H1) we have 
$$\|\widehat{D_z}\|
\leq   \sup_{n\in\mathbb{N}}\sum_{m=1}^\infty |D_z(n,m)| 
\leq   \sum_{m=1}^\infty  e^{-Re(z) \ell(m+k)} \leq   \sum_{m=1}^\infty  e^{-\epsilon \ell(m+k)} <1. \eqno(3.2)
$$
Writing $\ell^\infty(\mathbb C)$ as 
the corresponding direct sum of two subspaces,
we can then easily verify that  
$$
I -\widehat {M_z} = 
\left(
\begin{matrix}
I - \widehat {A_z} - \widehat {B_z} (I-\widehat {D_z})^{-1}\widehat {C_z}&-\widehat {B_z} (I-\widehat {D_z})^{-1}  \\ 0 &I 
\end{matrix}
\right)
\left(
\begin{matrix}
I &0  \\ -\widehat {C_z} & I - \widehat {D_z}
\end{matrix}
\right).\eqno(3.3)
$$

\noindent
Let us denote the $k \times k$ matrix   $W_z:= A_z +B_z (I-D_z)^{-1}C_z$, where each entry is given by a convergent series.
By (3.3), whenever $\det(I-W_z)\neq 0$ then  we see that $I-\widehat {M_z}$ is invertible,  with inverse
$$(I-\widehat {M_z})^{-1}=\left(
\begin{matrix}
I &0  \\ (I - \widehat {D_z})^{-1}\widehat {C_z} & (I - \widehat {D_z})^{-1}
\end{matrix}
\right)
\left(
\begin{matrix}
(I-\widehat {W_z})^{-1}&(I-\widehat {W_z})^{-1}\widehat {B_z} (I-\widehat {D_z})^{-1}  \\ 0 &I 
\end{matrix}
\right).\eqno(3.4)
$$
This leads to the following result.

\begin{lemma}\label{operator}
The operator $(I-\widehat {M_z})^{-1}$ has an analytic extension to $Re(z) > 0$
except when $\det(I-W_z) = 0$.
\end{lemma}

\begin{proof}
This follows from the identity (3.4) and since the $\epsilon > 0$ chosen in the above construction can be chosen arbitrarily small.
\end{proof}

\section{Complex functions}
We can now introduce a  complex function whose analytic properties will be useful in deriving our asymptotic estimates for $N_{\mathcal G}(x,R)$.   Fix $x \in \mathcal V$.
\begin{definition}
We can formally define the complex function $$\eta_{\mathcal G}(z)=\int_0^\infty e^{-zR}dN_{\mathcal G}(x,R)=\sum_{p\in P(x)} e^{-z\ell(p)}, \quad z \in \mathbb C,$$ where 
$P(x) = \{ p = e_1 \cdots e_n \hbox{ : } n \geq 0, i(e_1)=x \}$ is the set of paths in $\mathcal G$ starting at $x$. 
\end{definition}

We first observe that $\eta_{\mathcal G}(z)$ converges to an analytic function for $Re(z) > h$, by virtue of Definition \ref{gentropydef}.
In order to construct a meromorphic extension of $\eta_{\mathcal G}(z)$ we shall  relate $\eta_{\mathcal G}(z)$ to the matrix $M_z$.
For  $Re(z)>0$, we define:
\begin{enumerate}
\item[(a)]
$\underline{w}(z)=
 (\chi_{\mathcal E_x}(e_j)e^{-z\ell(j)})_{j=1}^\infty \in \ell^1(\mathbb C)$ where $\chi_{\mathcal E_x}$ denotes the characteristic function of  the set $\mathcal E_x = \{e \in \mathcal E \hbox{ : } i(e)=x\}$ of edges   whose  initial vertex  is $x$; and
\item[(b)] 
 $\underline {1} = (1)_{j=1}^\infty \in \ell^\infty(\mathbb C)$ is the vector all of whose entries are equal to $1$,  
\end{enumerate} 
  then we can formally   rewrite $\eta_{\mathcal G}(z)$ as
$$
\begin{aligned}
\eta_{\mathcal G}(z) 
=\sum_{p\in P(x)} e^{-z\ell(p)} &=
\underline {w}(z) \cdot \Big(\sum_{n=0}^\infty  \widehat{M_z}^n\Big) \underline 1\cr
&=
\underline {w}(z) \cdot \Big(I - \widehat{M_z}\Big)^{-1} \underline 1,
\end{aligned}
\eqno(4.1)
$$
where $w \cdot v = 
\sum_{j=1}^\infty w_j v_j$ for 
$w \in \ell^1(\mathbb C)$ and 
$v \in \ell^\infty(\mathbb C)$.  
Observe that for $Re(z) > 0$ we have 
$\underline {w}(z) \in \ell^1(\mathbb C)$ by (H1).
In particular, by Lemma \ref{operator} the expression in (4.1) extends  to $Re(z) > 0$,  and the locations of the poles are given by those $z$ such that the finite rank operator $(I-W_z)$ is not invertible. 
Moreover, we can easily write 
$$\eta_{\mathcal G}(z)=\frac{\phi(z)}{\det(I-W_z)}\eqno(4.2)$$
where $\phi(z)$ is holomorphic on $Re(z)> 0$.

\begin{prop}
$\eta_{\mathcal G}(z)$ has a meromorphic extension to $Re(z)>0$.
\end{prop}
\begin{proof}
Observe  that $\det(I-W_z)$ is the sum of a countable number of holomorphic functions which uniformly converge  on any compact domain 
in $Re(z)>0$ and hence 
$\det(I-W_z)$ is holomorphic.
The result follows from the identity (4.2).
\end{proof}
Let  $\epsilon<h$.
By (3.2) we can choose $k$ large enough such that  $(I-\widehat{D_z})$ is invertible,
 on the half plane $Re(z)\geq \epsilon$.
Recall that
a non-negative $n\times n$ matrix $M$ is {\it irreducible} if for all $i,j$ satisfying $1\leq i,j\leq n$ there exists a natural number $m$ such that $(M^m)_{i,j}>0$.

\begin{lemma}\label{lemnonneg}
Let $\sigma>0$. Then $W_\sigma$ is a non-negative irreducible  matrix. Furthermore,  $W_\sigma$ has a simple maximal positive eigenvalue
$\rho(\sigma)=\rho(W_\sigma)$, which depends  analytically on $\sigma$ and satisfies $\rho'(\sigma)<0.$

\end{lemma}

\begin{proof}
Recall that $W_\sigma=A_\sigma+B_\sigma(I-D_\sigma)^{-1}C_\sigma$, which by construction is a non-negative matrix. 
We can also deduce that  
 the matrix  $W_\sigma$ is irreducible. To see this, note that by assumption $\text{(H2)}$, for all $1 \leq i, j \leq k$, there exists some path of length $n$ starting with edge $e_i$ and ending with edge $e_j$. Such a path can be broken up into sub-paths of two types.
 The first type consists 
 of those paths that   stay completely within $\{e_1, \cdots, e_k\}$, and
the second type which consists of those paths that   initially enter the complement $\mathcal E - \{e_1, \cdots, e_k\}$ and finally leave at their end.  Note that $W_\sigma^n(i,j)$ is a sum including powers of $A_\sigma(i,j)$ (corresponding to sub-paths of the first type) and $B_\sigma (I-D_\sigma)^{-1}C_\sigma$ (corresponding to sub-paths of the second type), where the powers are less than or equal to $n$. Hence $W^n_\sigma (i,j)>0$.

 
We can now apply the  Perron-Frobenius theorem (see \cite{gant})
to deduce that the maximal positive eigenvalue   $\rho(\sigma)>0$
for  $W_\sigma$ exists and
that  $W_\sigma$ has associated  positive left and right eigenvectors $u(\sigma)$ and $v(\sigma)$ (which we normalise so that $u(\sigma)v(\sigma)=1$). By differentiating  the eigenvalue equations for $u(\sigma)$ and $v(\sigma)$, one can show that $$\rho'(\sigma)=u(\sigma)^TW'_\sigma v(\sigma)<0, $$
where $W'_\sigma$ is the matrix with entries  $W'_\sigma(i,j)=\frac{dW_x(i,j)}{dx}(\sigma)<0$ for all $i,j$ (see \cite{tuncel} for a similar argument).
\end{proof}

\begin{prop}
$h$ is a simple pole of $\eta_{\mathcal G}(z)$.
\end{prop}

\begin{proof}
For $z$ in a neighbourhood of $h$, we denote by
$\rho(z)$ the perturbed eigenvalue of $W_z$. We can write $\det(I-W_z)=(1-\rho(z))\Pi_{i=2}^k(1-\lambda_i(z))$, where the $\lambda_i(z)$ denote the other eigenvalues of $W_z$. 
Since the $\lambda_i(z)$ are bounded away from 1 for $z$ near $h$ (by the Perron-Frobenius theorem and standard perturbation theory), $\phi(h) \neq 0$
and  $\rho'(h)\neq 0$ (by Lemma \ref{lemnonneg}),
we can conclude that $(z-h)\eta_{\mathcal G}(z)$ converges to
a non-zero constant,  as $z$ tends to $h$.
\end{proof}

\begin{prop}
$\eta_{\mathcal G}(z)$ has no poles other than $h$ on the line $Re(z)=h$.
\end{prop}
\begin{proof}
Suppose for a contraction that  there exists another pole at $h+it$ ($t\neq 0$). Let $c$ be any closed path and choose an integer $k_c>k$ such that the edges of $c$ have index smaller than $k_c$. Then construct the $k_c\times k_c$ matrices $W_z$. From equation (4.2) we see that $\det(I-W_{h+it})=0$, and thus  
$1$ is an eigenvalue for $W_{h+it}$ and $W_h$. Furthermore, we  can see that $\rho(W_h)=1$ since otherwise  $\eta_{\mathcal G}(z)$ has a pole at $c>h$, contradicting Definition \ref{gentropydef}.

Next observe that $|W_{h+it}(a,b)|\leq W_h(a,b)$ for all $1\leq a,b \leq k$. Since $\rho(W_{h+it})\geq 1 =\rho(W_h)$, we can apply Wielandt's theorem (see \cite{gant}) which allows us to conclude that $\rho(W_{h+it})=\rho(W_h)=1$ and that there exists a diagonal matrix $D$, whose non-zero entries have unit modulus such that $W_{h+it}=DW_h D^{-1}$, and thus for all $n$ we have $W_{h+it}^n=DW_h^n D^{-1}$.

Suppose that the closed path $c$ begins with some edge $e_a$ and consists of $n$ edges. One can check that $W_{h+it}^n(a,a)=W_h^n(a,a)$ (since $W_{h+it}^n=DW_h^n D^{-1}$) and that $e^{(h+it)\ell(c)}$ is one of the terms in the left hand sum. However, this can only hold if $t$ is such that $\ell(c)t=2\pi m_c$ for some non-zero integer $m_c$.  As $c$ was arbitrary, the above construction implies that for all closed paths $c$, $\ell(c)\in d \mathbb N$ with $d = 2\pi/t$ which contradicts (H3).
\end{proof}

\section{Proof of Theorem \ref{asyminfinite}}

We can complete the proof using a similar approach to Parry in \cite{parry}, where he considered only finite matrices. 
 In particular, we will use the 
following formulation of the Ikehara--Wiener Tauberian theorem \cite{ellison} applied to our counting function, $N_{\mathcal G}(x,R)$. 

 \begin{thm}[Ikehara--Wiener Tauberian theorem]\label{tauberian}
Let $A: \mathbb R^+ \to \mathbb R^+$ be a 
 monotonic, non-decreasing function and formally
denote  $\eta(z):=\int_0^\infty e^{-zR}dA(R)$, for $z \in \mathbb C$. Then suppose that $\eta(z)$ has the following properties:
 
 \begin{enumerate}
     \item there exists some $a>0$ such that $\eta(z)$ is analytic on $Re(z) > a$;
     \item $\eta(z)$ has a meromorphic extension to a neighbourhood of the  half-plane $Re(z)\geq a$;
     \item $a$ is a simple pole for $\eta(z)$, i.e.,  $C=\lim_{\epsilon\searrow 0}(z-a)\eta(z) > 0$; and 
     \item the extension of $\eta(z)$ has no poles on the line $Re(z)=a$ other than $a$.   
     \end{enumerate}
 Then $A(R)\sim Ce^{aR}$ as $R\rightarrow +\infty$.

\end{thm}

From the results in the previous section we see that the $\eta_{\mathcal G}(s)$ satisfies the assumptions of Theorem \ref{tauberian} with $a=h$ and so we have proved Theorem \ref{asyminfinite}. 

\section{Translation surfaces}

In this section we will consider a definition of volume entropy for translation surfaces and prove asymptotic results using the work developed in the previous sections.

\begin{definition}
A translation surface $X$ is a compact surface with a flat metric
except at a finite set $\Sigma = \{x_1, \ldots, x_n\}$ of singular points with cone angles $2\pi (k(x_i)+1)$, where $k(x_i) \in \mathbb N$, for $i=1, \ldots, n$.
\end{definition}

A path which does not pass through singularities is a locally distance minimizing geodesic if it is a straight line segment. This includes geodesics which start and end at singularities, known as {\it saddle connections}. We will consider 
oriented saddle connections.

Geodesics can change direction if they go through a singular point, and a  pair of line segments ending and beginning, respectively,  at a singular point form a geodesic if the angle between them is at least $\pi$.
Thus a  locally distance minimising geodesic (of length $R$) on a translation surface $X$ with singularity set $\Sigma$, is a curve 
$\gamma:[0,R]\rightarrow X$ satisfying the following conditions:  
\begin{itemize}
\item There exist $0\leq t_1<...<t_n \leq R$, where $n\geq 0$, such that $\gamma(t_i)\in \Sigma$;
\item For $t_i<t<t_{i+1}$ $\gamma(t)\in X\backslash \Sigma$;
\item $\gamma:(t_i,t_{i+1}) \rightarrow X\backslash \Sigma$ is a geodesic segment (possibly a saddle connection);
\item 
The smallest angle between $\gamma|_{(t_{i-1},t_i)}$
and $\gamma|_{(t_i,t_{i+1})}$
is at least $\pi$ (cf. \cite{dankwart}, Lemma 2.1).
\end{itemize}

Let ${\mathcal{S}}=\{s_1,s_2,...\}$ be the set of oriented saddle connections 
ordered by non-decreasing lengths.

\begin{definition}
We define a saddle connection path $p
= (s_{i_1},...,s_{i_n})$  to be a finite string of oriented saddle collections $s_{i_1},...,s_{i_n}$ which form a geodesic path.


\end{definition}

We denote by $\ell(p) = \ell(s_1)+\ell(s_2)+\cdots +\ell(s_n)$ the sum of the lengths of the constituent saddle connections.
We let $i(p), t(p) \in \Sigma$
denote the initial and terminal singularities, respectively, of the 
saddle connection path  $p$.  

\begin{example}[Square tiled surfaces \cite{schmithusen}]
We can consider the square-tiled surfaces by identifying opposite sides of arrangements of a finite number of copies of the unit square (Figure \ref{fig:squaretiles}).
The values of the lengths of the saddle connections are of the form 
$\{\sqrt{n^2 + m^2} \hbox{ : } (n,m) \in \mathbb Z^2 - (0,0) \hbox{ co-prime}\}$.
\end{example}

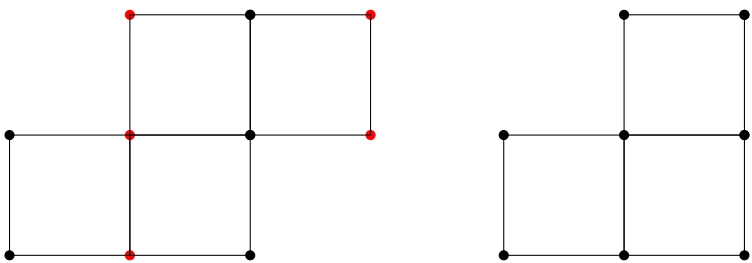
\begin{figure}[ht]
\begin{center}

\begin{tikzpicture}[scale=0.40]
  \draw [fill, black] (0,0) circle (0.15);
  \draw [fill, black] (0,4) circle (0.15);
    \draw [fill, black] (8,0) circle (0.15);
  \draw [fill, black] (8,4) circle (0.15);
    \draw [fill, black] (8,8) circle (0.15);
      \draw [fill, red] (4,0) circle (0.15);
      \draw [fill, red] (4,4) circle (0.15);
        \draw [fill, red] (4,8) circle (0.15);
  \draw [fill, red] (12,4) circle (0.15);
   \draw [fill, red] (12,8) circle (0.15);
     \draw [fill, black] (8,4) circle (0.15);
    \draw [fill, black] (8,8) circle (0.15);
  \draw (0,4) -- (4,4) -- (4,0) -- (0,0)--(0,4);
   \draw (4,4) -- (8,4) -- (8,0) -- (4,0)--(4,4);
 \draw (4,8) -- (8,8) -- (8,4) -- (4,4)--(4,8);
  \draw (8,8) -- (12,8) -- (12,4) -- (8,4)--(8,8);
 \end{tikzpicture}
 \hskip 1.5cm
 \begin{tikzpicture}[scale=0.40]
  \draw [fill, black] (0,0) circle (0.15);
  \draw [fill, black] (0,4) circle (0.15);
    \draw [fill, black] (8,0) circle (0.15);
  \draw [fill, black] (8,4) circle (0.15);
    \draw [fill, black] (8,8) circle (0.15);
      \draw [fill, black] (4,0) circle (0.15);
      \draw [fill, black] (4,4) circle (0.15);
        \draw [fill, black] (4,8) circle (0.15);
     \draw [fill, black] (8,4) circle (0.15);
    \draw [fill, black] (8,8) circle (0.15);
  \draw (0,4) -- (4,4) -- (4,0) -- (0,0)--(0,4);
   \draw (4,4) -- (8,4) -- (8,0) -- (4,0)--(4,4);
 \draw (4,8) -- (8,8) -- (8,4) -- (4,4)--(4,8);
 \end{tikzpicture}
 \end{center}
 \caption{(i) A square tiled surface formed from four tiles; (ii) a square tiled surface formed from  three tiles.}
 \label{fig:squaretiles}
\end{figure}

We now turn our attention to defining a notion of volume entropy for translation surfaces in terms of the growth of the volume of a ball as its radius tends to infinity. By analogy with the definition of volume entropy for Riemannian manifolds (Definition 1.1) we can consider the rate of growth of balls in the universal cover $\widetilde X$ of $X$.


\begin{definition}\label{voltrans}
Let $ \tilde x\in \widetilde X$ 
and consider a ball $B(\tilde x, R) \subset \widetilde X$ of radius $R>0$ with centre $\widetilde x$. 
 We define the volume entropy of $X$ to be
\[h=h(X):=\limsup_{R\rightarrow +\infty} \frac{1}{R}\log \hbox{\rm Vol}_{\widetilde X}(B(\tilde x, R))\] 
where $\hbox{\rm Vol}_{\widetilde X}$ denotes the natural volume on 
$\widetilde X$.
\end{definition}


Definition \ref{voltrans} is closely related to the definition of Dankwart
\cite{dankwart}, which was formulated in terms of orbital counting.
As in the case of the definitions of volume entropy for Riemannian manifolds and finite metric graphs, $h$ is independent of the choice of $\tilde x$.  For convenience, we can take $\tilde x$ to be the lift of a singularity $x \in \Sigma$.

It is useful to interpret this definition in terms of $X$ rather than $\widetilde X$.
To this end we have the following definition.

\begin{definition}\label{dank}
Let $m_{R}(y)$ be number of distinct geodesic arcs in $X$ from $x$ to $y$ of length at most $R$.
\end{definition}

We can now rewrite
$\hbox{\rm Vol}_{\widetilde X}(B(\tilde x, R)) =\int_X m_R(y) d\hbox{\rm Vol}_{X}(y)$
(see Figure 3).   For economy of notation we  will write $V(x, R):= \hbox{\rm Vol}_{\widetilde X}(B(\tilde x, R))$.\\

\begin{figure}[ht]\label{overlap}
\begin{center}
\begin{tikzpicture}[scale=0.8]
  \draw [fill, black] (2,-2) circle (0.1);
   \draw [fill, black] (-2,2) circle (0.1);
    \draw [fill, black] (2,2) circle (0.1);
     \draw [fill, black] (-2,-2) circle (0.1);
  \draw (-4,2) -- (-2,2) -- (-2,4) -- (2,4) -- (2,2)--(4,2)--(4,-2)--(2,-2)--(2,-4)--(-2,-4)--(-2,-2)--(-4,-2)--(-4,2);
   \draw [red,thick,domain=-90:180] plot ({2-1.5*cos(\x)}, {2 -1.5*sin(\x)});
 \draw [pink,thick,domain=-90:180] plot ({-2+1.5*cos(\x)}, {-2 +1.5*sin(\x)});
  \draw [orange,thick,domain=-180:90] plot ({-2+1.5*cos(\x)}, {2 +1.5*sin(\x)});
   \draw [magenta,thick,domain=0:270] plot ({2+1.5*cos(\x)}, {-2 +1.5*sin(\x)});
 \end{tikzpicture}
 \hskip 1.5cm
\begin{tikzpicture}[scale=0.8]\label{Overlap1}
  \draw [fill, black] (2,-2) circle (0.1);
   \draw [fill, black] (-2,2) circle (0.1);
    \draw [fill, black] (2,2) circle (0.1);
     \draw [fill, black] (-2,-2) circle (0.1);
  \draw (-4,2) -- (-2,2) -- (-2,4) -- (2,4) -- (2,2)--(4,2)--(4,-2)--(2,-2)--(2,-4)--(-2,-4)--(-2,-2)--(-4,-2)--(-4,2);
   \draw [red,thick,domain=-35:125] plot ({2-3.5*cos(\x)}, {2 -3.5*sin(\x)});
 \draw [pink,thick,domain=-35:125] plot ({-2+3.5*cos(\x)}, {-2 +3.5*sin(\x)});
  \draw [orange,thick,domain=-125:35] plot ({-2+3.5*cos(\x)}, {2 +3.5*sin(\x)});
   \draw [magenta,thick,domain=55:215] plot ({2+3.5*cos(\x)}, {-2 +3.5*sin(\x)});
      \node at (0,0) {$4$};
      \node at (3,0) {$2$};
            \node at (-3,0) {$2$};
                    \node at (0,3) {$2$};
                            \node at (0,-3) {$2$};
  \node at (0.8,0.8) {$3$};  
   \node at (0.8,-0.8) {$3$}; 
    \node at (-0.8,0.8) {$3$}; 
     \node at (-0.8,-0.8) {$3$};                          
 \end{tikzpicture}
 \end{center}
 \caption{(i) A small ball centred at a singularity; (ii) As the radius $R$ increases the ball overlaps with itself (and the values of the multiplicity function $m_R(\cdot)$ are indicated).}
 \end{figure}
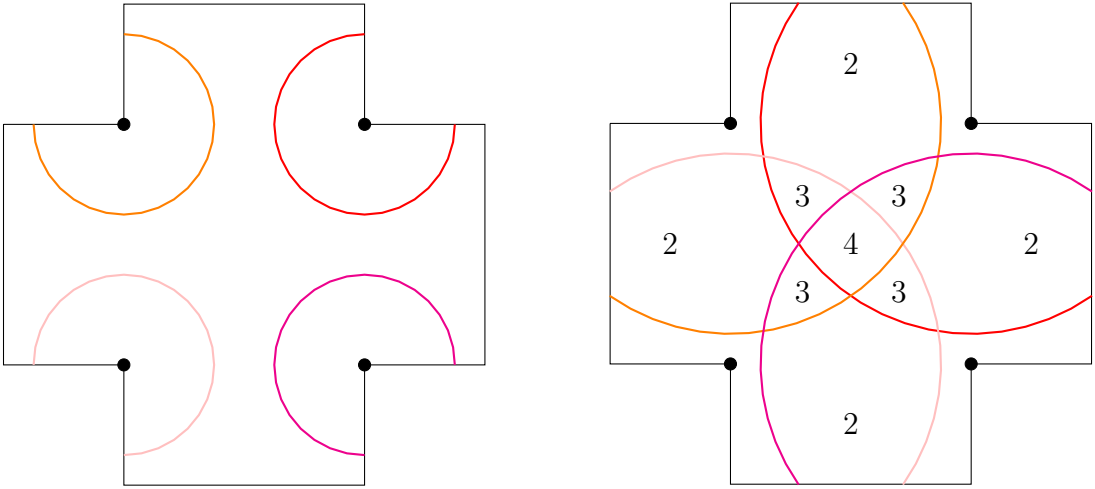

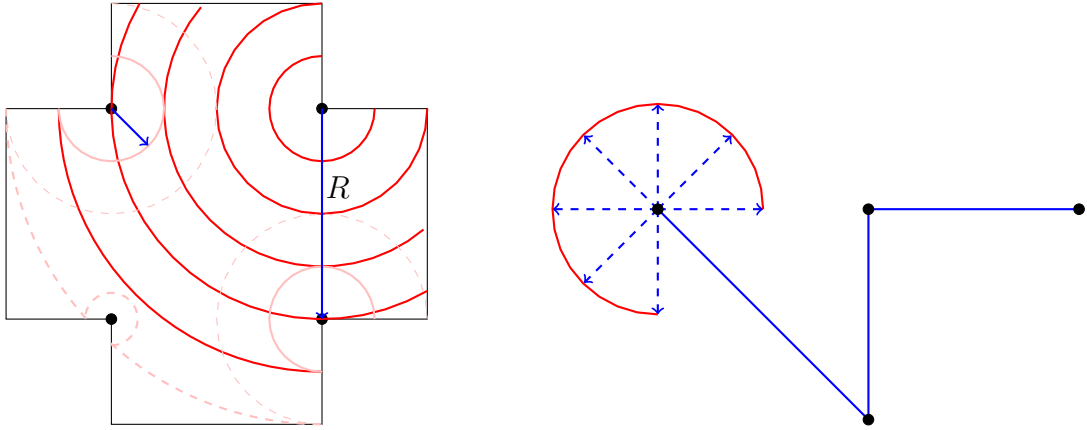
\begin{figure}[ht]\label{geopic}
\begin{center}
\begin{tikzpicture}[scale=0.7]
  \draw [fill, black] (2,-2) circle (0.1);
   \draw [fill, black] (-2,2) circle (0.1);
    \draw [fill, black] (2,2) circle (0.1);
     \draw [fill, black] (-2,-2) circle (0.1);
  \draw (-4,2) -- (-2,2) -- (-2,4) -- (2,4) -- (2,2)--(4,2)--(4,-2)--(2,-2)--(2,-4)--(-2,-4)--(-2,-2)--(-4,-2)--(-4,2);
   \draw [pink,thick,dashed, domain=48:90] plot ({2-6.0*cos(\x)}, {2 -6.0*sin(\x)});
    \draw [pink,thick,dashed, domain=0:42] plot 
    ({2-6.0*cos(\x)}, {2 -6.0*sin(\x)});
        \draw [red,thick, domain=0:90] plot ({2-5.0*cos(\x)}, {2 -5.0*sin(\x)});
      \draw [red,thick,domain=-30:120] plot ({2-4.0*cos(\x)}, {2 -4.0*sin(\x)});
     \draw [red,thick,domain=-40:130] plot ({2-3.0*cos(\x)}, {2 -3.0*sin(\x)});
   \draw [red,thick,domain=-90:180] plot ({2-2.0*cos(\x)}, {2 -2.0*sin(\x)});
      \draw [red,thick,domain=-90:180] plot ({2-1.0*cos(\x)}, {2 -1.0*sin(\x)});
 \draw[pink,thick,domain=90:-180] plot ({2-1.0*cos(\x)}, {-2 -1.0*sin(\x)});
 \draw[pink,thick,domain=270:0] plot ({-2-1.0*cos(\x)}, {2 -1.0*sin(\x)});
 \draw[pink,dashed,domain=90:-180] plot ({2-2.0*cos(\x)}, {-2 -2.0*sin(\x)});
 \draw[pink,dashed, domain=270:0] plot ({-2-2.0*cos(\x)}, 
 {2 -2.0*sin(\x)});
  \draw [pink,thick,dashed,domain=-270:0] plot ({-2-0.5*cos(\x)}, {-2 -0.5*sin(\x)});
 \draw[->,thick, blue] (2,2) -- (2,-2);
 \draw[->,thick, blue] (-2,2) -- (-1.3,1.3);
 \node at (2.3,0.5) {$R$};
 \end{tikzpicture}
  \hskip 1.5cm
 \begin{tikzpicture}[scale=0.7]
 \draw[-,thick, blue] (2,2) -- (2,-2);
  \draw[-,thick, blue] (2,-2) -- (-2,2);
 \draw[->,thick, dashed, blue] (-2,2) -- (-2,0);
  \draw[->,thick, dashed, blue] (-2,2) -- (-2,4);
   \draw[->,dashed,thick, blue] (-2,2) -- (0,2);
   \draw[->,dashed,thick, blue] (-2,2) -- (-4,2);
   \draw[->,dashed,thick, blue] (-2,2) -- (-3.4,0.6);
    \draw[->,dashed,thick, blue] (-2,2) -- (-3.4,3.4);
    \draw[->,dashed,thick, blue] (-2,2) -- (-0.6,3.4);
   \draw[-,thick, blue] (2,2) -- (6,2);
    \draw [fill, black] (2,2) circle (0.1);
    \draw [fill, black] (6,2) circle (0.1);
    \draw [fill, black] (2,-2) circle (0.1);
    \draw [fill, black] (-2,2) circle (0.1);
    \draw [red,thick,domain= -180:90] plot ({-2-2.0*cos(\x)}, {2 -2.0*sin(\x)});
 \end{tikzpicture}
 \end{center}
 \caption{(i) The radii of $B(x,R)$ are concatenations of saddle connections followed by a radial line segment from a singularity; (ii) A heuristic figure illustrating that the boundary of $B(x,R)$ will consist of the union of circular arcs centred on singularities reached via concatenations of saddle connections}
 \end{figure}


Let $x\in \Sigma$ be a singularity, then we define 
$$\pi(x,R) := 
\left\{p  \hbox{ : } 
i(p) = x \hbox{ and }l(p) \leq R
\right\}
 $$
 to be the number of saddle connection paths starting at $x$ of length less than or equal to $R$.

\begin{lemma}
Let $X$ be a translation surface and fix a singularity $x\in \Sigma$ and let $2\pi(k(x)+1)$ be the cone angle of $x$. Then for $R>0$,
\[V(x,R)=(k(x)+1)\pi R^2 + \sum_{p\in \pi(x,R)} k(t(p))\pi (R-\ell(p))^2,  \]
where  the singularity at the end of path $p$ has cone angle $2\pi (k(t(p))+1)$.
\end{lemma}

\begin{proof}
The volume contributed by the geodesics starting from $x$ which do not pass through a singularity is given by $(k(x)+1)\pi R^2$, where $2\pi (k(x)+1)$ is the cone angle at $x$.
On the other hand, the contribution to the volume by those geodesics $\gamma$
which pass through one or more singularities
comes when the  geodesic leaves its last singularity at time $\ell({p}) < R$, say.
It can exit in one of $2\pi k(p)$ directions.  Then the total volume of such $\gamma$ is given by $k(t(p))\pi  (R-(\ell(p))^2$. 
\end{proof}

\noindent
We shall now prove asymptotic results for translation surfaces using the analysis developed for infinite graphs.\\

\begin{definition}
 We can associate to $X$ the countable matrix $M_0$, indexed by $\mathcal S$, 
 defined by 
$$M_0(s,s')=\begin{cases} 1 &\mbox{if }  \text{ss$'$ form a saddle connection path}, \\
0 & \mbox{otherwise. }  \end{cases} $$
For each  $z\in \mathbb C$ 
we define the matrix $M_z$ by $M_z(s,s')=M_0(s,s')e^{-z\ell(s')}$ for $s,s'\in \mathcal S$.\\
\end{definition}
In order that the matrices have the same properties that served us well for graphs, we require  specific  features  of a  translation surface.

\medskip
\noindent
{\bf Translation Hypotheses.}
Henceforth, we shall consider translation surfaces whose countable set of saddle connections is denoted by $\mathcal S$.  Moreover, we require that $\mathcal S$ and the lengths of the saddle connections satisfy the following properties:
\begin{enumerate}
\item[(T1)]
For all $\sigma > 0$ we have $\sum_{s \in \mathcal S} 
e^{-\sigma \ell(s)} < +\infty$; 
\item[(T2)] For any directed saddle connections $s, s' \in \mathcal S$ there exists a saddle connection path beginning with $s$ and ending with $s'$; and
\item[(T3)] There does not exist a
$d > 0$ such that 
$$
\{ \ell (c) \hbox{ : } c \hbox{ is a closed saddle connection path} \} \subset d \mathbb N.
$$
\end{enumerate}
We claim that the above hypotheses hold for all translation surfaces.\\

Property (T1) follows from the lower bound in  following result (see \cite{masurlower} and \cite{masurupper}).

\begin{prop}\label{growththm}
Let $X$ be a translation surface and let $N(X,L)$ denote the number of saddle connections on $X$ of length less than or equal to $L$. Then there exists constants $0 < c_1<c_2<\infty$ such that
\[c_1L^2\leq N(X,L)\leq c_2L^2,\]
for $L$ sufficiently large.
\end{prop}

To see that Hypotheses (T2) and (T3) hold for all translation surfaces, we require the following result in \cite{dankwart} which we restate for our purposes here.
\begin{prop}\label{dankg}
Let $X$ be a translation surface. 
If $s,s' \in \mathcal S$ are oriented saddle connections then  
there exists a saddle connection path which starts with $s$ and ends with $s'$. 
\end{prop}
Hypothesis (T2) follows immediately from this fact. 

To show 
Hypothesis (T3) holds for all surfaces
we first note
 that if the lengths of all closed geodesics were an integer multiple of some constant $d$, then the  length of every saddle connection would be an integer multiple of $d/2$. To see this, let $s$ be any saddle connection on $X$. If $i(s)=t(s)$ then $s$ is a closed geodesic and so we are done. If $i(s)\neq t(s)$ then by Proposition \ref{dankg}, there exists a closed saddle connection path $c_i$ such that $c_i$ passes through $i(s)$ and that $\bar{s}c_is$ forms a saddle connection path (where $\bar{s}$ is the saddle connection $s$ with reversed orientation). Similarly, there exists a closed saddle connection path $c_t$ which starts and ends at $t(s)$, such that $sc_t\bar{s}$ forms a saddle connection path. Note that the concatenation $sc_t\bar{s}c_i$ is also a closed saddle connection path
of length $2 \ell (s) + \ell(c_t) + \ell(c_i)$ 
  and so by Hypothesis (T3), $\ell(s)\in (d/2)\mathbb N$.
Let us now assume for a contradiction that (T3) does not hold and, in particular, the above property holds for the saddle connection lengths.

Using results in \cite{smillie}, $X$ contains an embedded cylinder $C$ (the product of a circle with  an interval) whose boundaries consist of a single saddle connection or multiple parallel saddle connections.  We now aim to construct a countable family of triangles using this cylinder (Figure \ref{triangles}).
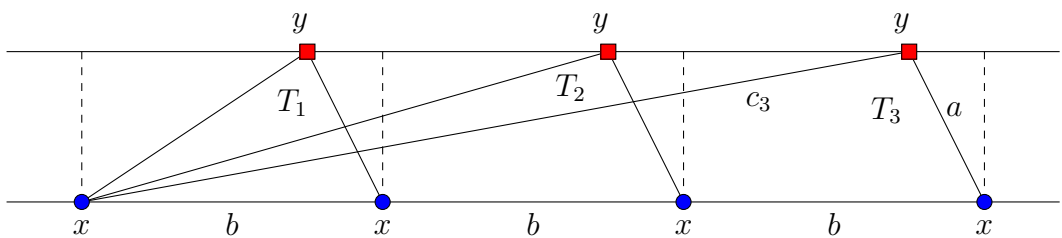
\begin{figure}[h]
\begin{center}
\begin{tikzpicture}[scale =1]
    \draw (0,0) --(11,2);
    \draw (0,0) --(7,2);
    \draw (0,0) --(3,2);
    \draw (4,0) --(3,2);
    \draw (8,0) --(7,2);
    \draw (12,0) --(11,2);
    \draw (-1,0) --(13,0);
    
    \draw (-1,2) --(13,2);
    
    \draw [dashed] (0,0) --(0,2);
    \draw [dashed] (12,0) --(12,2);
    \draw [dashed] (4,0) --(4,2);
    \draw [dashed] (8,0) --(8,2);

    \draw [fill=red] (-0.1+3,-0.1+2) rectangle (0.1+3,0.1+2);
    \draw [fill=red] (-0.1+7,1.9) rectangle (0.1+7,2.1);
    \draw [fill=red] (-0.1+11,1.9) rectangle (0.1+11,2.1);
    \draw [fill=blue] (0,0) circle [radius = 0.1];
    \draw [fill=blue] (4,0) circle [radius = 0.1];
    \draw [fill=blue] (8,0) circle [radius = 0.1];
    \draw [fill=blue] (12,0) circle [radius = 0.1];
  \node[below] at (0,-0.1) {$x$};
   \node[below] at (4,-0.1) {$x$};
   \node[below] at (8,-0.1) {$x$};
   \node[below] at (12,-0.1) {$x$};
   \node[above] at (-0.1+3,2.1) {$y$};
   \node[above] at (-0.1+7,2.1) {$y$};
   \node[above] at (-0.1+11,2.1) {$y$};
  \node[below] at (2,0) {$b$};
    \node[below] at (6,0) {$b$};
        \node[below] at (10,0) {$b$};
  \node[above] at (11.6,1) {$a$};
  \node[above] at (2.8,1.0) {$T_1$};
  \node[above] at (6.5,1.2) {$T_2$};
    \node[above] at (10.7,0.9) {$T_3$};
        \node[below] at (9.0,1.6) {$c_3$};
    
\end{tikzpicture}
\end{center}
\caption{Three copies of a cylinder on $X$ with two singularities on separate boundaries represented by 
circles and 
squares. The corresponding triangles $T_1$, $T_2$ and $T_3$ are also drawn.} 
\label{triangles}
\end{figure}
Fix two singularities $x$ and $y$, one from each boundary. Let $b$ denote the union of saddle connections which form the boundary of the cylinder connecting $x$ to itself. Let $a$ be one of the 
 saddle connection connecting $x$ to $y$ across the cylinder such that the angle between $a$ and $b$ is acute.  Then consider the unique saddle connection $c_n$ connecting $x$ to $y$ which is defined to be the third side in a triangle $T_n$ whose other  edges are   $b$ concatenated with itself  $n$ times and a. 
  By hypothesis each edge has length which is an integer multiple of $d/2$.
  However, by elementary Euclidean geometry we can show that this cannot hold for all sufficiently large $n$, giving the required contradiction.
\\

To derive an asymptotic estimate 
for $V(x,R)$ we  can associate the complex function
$$
\eta_X(z) = \int_0^\infty e^{-zR} dV(x,R).
$$
Let $\mathcal{P}(x) := 
\left\{p  \hbox{ : } 
i(p) = x \right\}$ denote the set of saddle connection paths starting at $x$. We can rewrite $\eta_X(z)$ as follows:
$$
\begin{aligned}
\eta_X(z)&= \int_0^\infty e^{-zR} 
\left(\frac{d}{dR}\hbox{\rm Vol}(B(x,R))\right) dR\cr
&= 
2\pi (k(x)+1) \int_0^\infty e^{-zR} R  dR
+ 2\pi  \sum_{p\in\mathcal{P}(x)}
k(t(p))
\int_{\ell(p)}^\infty e^{-zR} (R - \ell(p)) dR\cr
&= 
2\pi (k(x)+1) \int_0^\infty e^{-zR} R  dR
+ 2\pi  \sum_{p\in\mathcal{P}(x)}
k(t(p))
e^{- z \ell(p)}\int_{0}^\infty e^{-zR}R dR\cr
&= 
\frac{2\pi}{z^2} (k(x)+1)
+ \frac{2\pi}{z^2}  \sum_{p\in\mathcal{P}(x)}
k(t(p))
 e^{-z\ell(p)} \cr
 &= \frac{2\pi}{z^2} (k(x)+1)
+ \frac{2\pi}{z^2} 
 {\underline v}(z) 
 \cdot
 ( I - \widehat{M_z})^{-1}{\underline u},
\end{aligned} \eqno(6.1)
$$
where 
 $\underline u 
 =  (k(t(s_j)))_{j=1}^\infty \in \ell^\infty(\mathbb C)$
 and $\underline v(z) 
 =  
 (\chi_{\mathcal E_x}(s_j)
 e^{-z\ell (s_j)})_{j=1}^\infty
 \in \ell^1(\mathbb C)$, where 
 $\chi_{\mathcal E_x}$ denotes the characteristic function of the set 
 $\mathcal E_x = \{s \in \mathcal S \hbox{ : } i(s) =x\}$
 of saddle connections starting from the singularity $x \in \Sigma$.
 

\medskip
\begin{lemma}
 The function $\eta_X(z)$ is analytic for $Re(z) > h$ and has a meromorphic extension to $Re(z) > 0$.  Moreover $\eta_X(z)$ has a simple pole at $z=h$ and no other 
 poles on $Re(z) = h$.
\end{lemma}

\begin{proof}
We can apply the analysis of 
$( I - \widehat{M_z})^{-1}$ in Section 4 to (6.1), where 
we use hypotheses (T1)-(T3) in place of (H1)-(H3).
\end{proof}

We can now apply Theorem \ref{tauberian}
to deduce the following.

\begin{thm}\label{trans}  There exists a $C > 0$ such that $V(x,R)
\sim C e^{hR}$ as $R \to +\infty$, i.e., 
$$
\lim_{R \to +\infty}
\frac{V(x,R)}{e^{h R}} = C.
$$

\end{thm}
Typically $C = C(x)$ will depend on the choice of $x$.

There is a  closely related result
for counting 
the number of geodesic arcs $N_X(x, y,R)$
starting  at $x \in \Sigma$ and finishing at $y \in \Sigma$.

\begin{prop}\label{points}
There exists a $D > 0$ such that 
$N_X(x,y,R)
\sim D e^{hR}$ as $R \to +\infty$, i.e., 
$$
\lim_{R \to +\infty}
\frac{N_X(x,y,R)}{e^{h R}} = D.
$$
\end{prop}

\begin{proof}
The proof simply requires replacing the function $\eta_X(z)$ by the function
$$
\eta_N(z) = \int_0^\infty e^{-zR} dN_X(x,R) =  
 {\underline v}(z) \cdot( I - \widehat{M_z})^{-1}{\underline w},$$
where ${\underline w} = (\chi_{\mathcal F}(s_i))_{i=1}^\infty$, with $\chi_{\mathcal F}(s)$ denoting the characteristic function 
for the set $\mathcal F = \{s \in \mathcal{S} \hbox{ : } t(s_i)=y\}$ of saddle connections ending at the singularity $y$
and $u(z)$ was defined after equation (6.1).  Again the properties of $(I-\widehat{M_z})^{-1}$ allow one to apply Theorem
\ref{tauberian} to deduce the result.
\end{proof}

\begin{rem}
We conclude with some final remarks.
\begin{enumerate}
\item
It is not necessary for the ball in Theorem \ref{trans} to be centered at a singularity.
Let $y\in X - \Sigma$ and let 
$G$ be the set of geodesics $g$, from $y$ to a singularity, such that $g$ has length $\ell(g)$. Order $G$ by non-decreasing lengths.
 We define a matrix $P$
where 
\begin{enumerate}
\item
the rows are indexed by such geodesics $g$ and the columns are indexed by the oriented saddle connections $s$; 
\item
the non-zero entries correspond to pairs
$g$, $s$ such that:
\begin{enumerate}
\item The singularity $t(g)$ at the end of $g$ is the same as that $i(s)$
at the  start of the saddle connection $s$; and
\item The geodesic $g$ and saddle connection $s$ have an angle of at least $\pi$ between them.
\end{enumerate}
\item The non-zero entries are  $P_z(g,s) = e^{-z \ell (s)}$.
\end{enumerate}
 One can then modify the complex function 
 to
$\eta_X(z) =  \frac{2\pi}{z^2} (k(x)+1)
+ \frac{2\pi}{z^2} 
 {\underline v_p(z)} \cdot \widehat{P_z} ( I - \widehat{M_z})^{-1}\underline u$, where $\underline v_p(z)=(
 e^{-z\ell (g)})_{g\in G}
 \in \ell^1(\mathbb C)$ and then continue the proof as in 
 Theorem \ref{trans}.
\item
Theorem \ref{trans} also follows as a corollary of Theorem \ref{points} by using a simple approximation argument.
In particular, this shows that $C = D \int_0^\infty e^{-u} u^2 du$.
\item
Let 
$L(x,R)$ be the total circumference of a circle centred at $x$ and whose radius is a geodesic of length $R$.
The same approach as in the proof of Theorem \ref{trans} (or an approximation argument as in item 2  would give an asymptotic formula of the form $L(x,R) \sim E e^{hR}$, as $R \to +\infty$.
\item Eskin and Rafi have announced a closely related asymptotic result 
to Theorem \ref{trans}
for 
closed geodesics on $X$.  By studying zeta functions $\zeta_X(z)$ instead of eta functions $\eta_X(z)$ they show that the number of closed geodesics of length at
 most $R>0$ is asymptotic to $e^{hR}/hR$
 as $R \to +\infty$.


\end{enumerate}
\end{rem}


\begin{thebibliography}{111}

\bibitem{dankwart}
K. Dankwart, 
On the large scale geometry of flat surfaces,
Dissertation, Bonn, 2010.

 
\bibitem{ellison} 
W. and F. Ellison, {\it Prime numbers}, Hermann, 1985.  

\bibitem{gant}
F. Gantmacher, {\it The theory of matrices}, Chelsea, New York, 1959.

 
\bibitem{hofbauerkeller}
F. Hofbauer and G. Keller,
Zeta-functions and transfer operators for piecewise linear transformations, 
Journal f\"ur die Reine und Angewandte Mathematik,
352 (1984) 100-113.


\bibitem{hubert} 
P. Hubert and T. Schmidt,  An introduction to Veech surfaces,  in {\it Handbook of Dynamical Systems}
(ed.s B. Hasselblatt and A. Katok), Elsevier,  Vol. 1B (2005)  pp 501-526.

\bibitem{lim} 
S. Lim, Minimal volume entropy on graphs, 
Trans. Amer. Math. Soc.,
360 (2008) 5089-5100.
.
\bibitem{manning}
A. Manning, Topological entropy for geodesic flows, 
Annals of Math., 110 (1979) 567-573.  

\bibitem{margulis-announcment}
G. Margulis, Applications of ergodic theory to the investigation of manifolds of negative curvature, Funktsional. Anal. i Prilozhen., 3:4 (1969), 89--90.

\bibitem{margulis}
G. Margulis, {\it On Some Aspects of the Theory of Anosov Systems}, Springer, Berlin, 2004.

\bibitem{masurlower}
H. Masur, Lower bounds for the number of saddle connections and closed trajectories
of a quadratic differential, Holomorphic functions and moduli, Vol. I (Berkeley, CA,
1986), Math. Sci. Res. Inst. Publ., vol. 10, Springer, New York, 1988, pp. 215-228.

\bibitem{masurupper}
H. Masur,
The growth rate of trajectories of a quadratic differential, Ergodic Theory and
Dynamical Systems 10 (1990) 151-176.

\bibitem{smillie}
H. Masur and J. Smillie, Hausdorff dimension of sets of nonergodic measured foliations, Ann. of Math. (2), 134 (1994) 455-543.

\bibitem{parry}
W. Parry, 
An analogue of the prime number theorem for closed orbits of shifts of finite type and their suspensions, 
Israel J. Math., 45 (1983) 41-52.

\bibitem{tuncel}
W. Parry and S. Tuncel,
{\it Classification Problems in Ergodic Theory},
LMSLNS 67, C.U.P., 
Cambridge.

\bibitem{PU} M. Pollicott
and M. Urbanski, 
Asymptotic counting in conformal dynamical systems, Memoirs Amer. Math. Soc., to appear.


\bibitem{schmithusen}
G. Schmith\"usen,
Examples of origamis. Proceedings of the III Iberoamerican Congress on Geometry. In: 
{\it The Geometry of Riemann Surfaces and Abelian Varieties}. Contemp. Math. 397, 2006, 193-206.



\bibitem{zorichflat} 
A, Zorich, Flat Surfaces, in {\it Frontiers in Number Theory, Physics and Geometry I} (eds. P. Cartier, B. Julia, P. Moussa, P. Vanhov), Springer, Berlin 2006, 439-586.


\end{thebibliography}
\end{document}